\documentclass[12pt, reqno]{amsart}

\usepackage[a4paper, centering, total={170mm,240mm}]{geometry}

\usepackage{amssymb,amsmath,latexsym,amsthm,amsfonts}
\usepackage{blindtext}
\usepackage{todonotes}
\usepackage{esint,dsfont}

\def\R{{\mathbb R}}
\def\N{\mathbb{N}}
\def\C{\mathbb{C}}
\def\Z{\mathbb{Z}}
\def\D{\mathbb{D}}

\def\F{\mathbb{F}}
\def\ii{\mathrm{i}}

\newtheorem{prop}{\bf Proposition}[section]
\newtheorem{thm}[prop]{\bf Theorem}
\newtheorem{cor}[prop]{\bf Corollary}
\newtheorem{lem}[prop]{\bf Lemma}
\newtheorem{rmk}[prop]{\it Remark}

\begin{document}

\title[$M_{d}$-AP and Unitarisability]{{\bf\Large The $M_d$-Approximation Property and Unitarisability}}

\author[I. Vergara]{Ignacio Vergara}
\address{Saint-Petersburg State University, Leonhard Euler International Mathematical Institute,
14th Line 29B, Vasilyevsky Island, St. Petersburg, 199178, Russia}

\email{ign.vergara.s@gmail.com}
\thanks{This work is supported by the Ministry of Science and Higher Education of the Russian Federation, agreement № 075-15-2022-287}

\makeatletter
\@namedef{subjclassname@2020}{%
  \textup{2020} Mathematics Subject Classification}
\makeatother



\begin{abstract}
We define a strengthening of the Haagerup--Kraus approximation property by means of the subalgebras of Herz--Schur multipliers $M_d(G)$ ($d\geq 2$) introduced by Pisier. We show that unitarisable groups satisfying this property for all $d\geq 2$ are amenable. Moreover, we show that groups acting properly on finite-dimensional CAT(0) cube complexes satisfy $M_d$-AP for all $d\geq 2$. We also give examples of non-weakly amenable groups satisfying $M_d$-AP for all $d\geq 2$.
\end{abstract}

\begingroup
\def\uppercasenonmath#1{} 
\let\MakeUppercase\relax 
\maketitle
\endgroup

\section{{\bf Introduction}}
Let $\pi:G\to\mathcal{B}(H)$ be a \textit{uniformly bounded} representation of a group $G$ on a Hilbert space $H$. This means that
\begin{align*}
\sup_{t\in G}\|\pi(t)\| < \infty.
\end{align*}
We say that $\pi$ is \textit{unitarisable} if there exists an invertible operator $S\in\mathcal{B}(H)$ such that the representation $S\pi(\,\cdot\,)S^{-1}$ is unitary, i.e., $\|S\pi(t)S^{-1}\|=1$ for all $t\in G$. We say that $G$ is \textit{unitarisable} if every uniformly bounded representation $\pi:G\to\mathcal{B}(H)$ is unitarisable. Day \cite{Day} and Dixmier \cite{Dix} independently proved that all amenable groups are unitarisable. The converse (known now as the Dixmier problem) remains open.

\noindent {\bf Dixmier's question:} Is every unitarisable group amenable?

We refer the reader to \cite{Pis2} and \cite{Pis} for very nice introductions to these topics. Let us point out that considerable progress has been made in more recent years regarding this question, by means of a wide ranging of tools, including $L^2$-Betti numbers \cite{EpsMon, Osi}, random embeddings \cite{MonOza,Alp} and Littlewood functions \cite{GGMT}.

In \cite{Pis}, Pisier introduced a family of algebras $M_d(G)$ ($d\geq 1$) associated to a group $G$, such that $M_2(G)$ coincides with the Herz--Schur multiplier algebra of $G$. These spaces satisfy
\begin{align*}
B(G)\subset M_{d+1}(G) \subset M_d(G) \subset M_1(G) = \ell_\infty(G),\quad (d\geq 1).
\end{align*}
Here $B(G)$ stands for the Fourier--Stieltjes algebra of $G$, i.e., the space of coefficients of unitary representations of $G$. Bo\.{z}ejko proved \cite{Boz} that $G$ is amenable if and only if $M_2(G)=B(G)$. Observe that in this case all the algebras $M_d(G)$ ($d\geq 2$) are the same. Pisier showed \cite[Theorem 2.9]{Pis} that something similar happens for unitarisable groups.

\begin{thm}[Pisier]\label{Thm_Pis}
Let $G$ be a unitarisable group. Then there exists $d\geq 1$ such that $M_d(G)=B(G)$.
\end{thm}

In \cite{HaaKra}, Haagerup and Kraus defined a very weak form of amenability for locally compact groups, which they called the approximation property (AP). In the present work, we shall only focus on discrete groups. We say that $G$ has the AP if the constant function $1$ is in the $\sigma(M_2(G),X_2(G))$-closure of the group algebra $\C[G]\subset M_2(G)$. Here $X_2(G)^*=M_2(G)$; see \S\ref{Sec_Gen_mult} for details. This property is indeed much weaker than amenability, and in fact, there were no known concrete examples of groups without AP until the work of Lafforgue and de la Salle \cite{LafdlS}, where they showed, among many other things, that $\operatorname{SL}(3,\Z)$ does not satisfy AP. This led to a very fruitful series of articles \cite{HaadLa,HaadLa2} culminating with a complete characterisation of connected Lie groups without AP \cite{HaKndL}.

It turns out that the algebras $M_d(G)$ are dual spaces too, which allows us to extend the definition of the AP to these new spaces of multipliers. Let $d\geq 1$. We say that $G$ has the $M_d$-AP if the constant function $1$ is in the $\sigma(M_d(G),X_d(G))$-closure of $\C[G]\subset M_d(G)$, where $X_d(G)^*=M_d(G)$. We will see that $M_{d+1}$-AP implies $M_d$-AP, but we do not know whether these properties are different for $d\geq 2$. On the other hand, since $M_1(G)=\ell_\infty(G)$, every group has $M_1$-AP, so clearly it is different from $M_2$-AP, which is the AP of Haagerup and Kraus.

Our first result gives a positive answer to the Dixmier problem for groups satisfying $M_d$-AP for every $d\geq 1$.

\begin{thm}\label{Thm_1}
Let $G$ be a group satisfying $M_d$-AP for all $d\geq 1$. If $G$ is unitarisable, then it is amenable.
\end{thm}

This raises the question of how restrictive it is for a group to satisfy $M_d$-AP for every $d\geq 1$. We are far from having a complete answer to this question, but we can provide a sufficient condition coming from geometric group theory. In \cite{GueHig}, Guentner and Higson proved that groups acting properly on finite-dimensional CAT(0) cube complexes are weakly amenable with Cowling--Haagerup constant 1. We refer the reader to their paper for details on CAT(0) cube complexes. This result was obtained independently by Mizuta \cite{Miz} using different methods.

A group $G$ is said to be weakly amenable if there is $C\geq 1$ such that the constant function $1$ is in the $\sigma(M_2(G),X_2(G))$-closure of the set
\begin{align*}
\{\varphi\in\C[\Gamma]\ :\ \|\varphi\|_{M_2}\leq C\}.
\end{align*}
The Cowling--Haagerup constant of $G$ is the infimum of all $C\geq 1$ such that the condition above holds. Clearly, weak amenability implies AP, but the converse does not hold. Indeed, Haagerup showed \cite{Haa} that the group $\operatorname{SL}(2,\Z)\ltimes\Z^2$ is not weakly amenable. On the other hand, it has the AP because the AP is stable under semidirect products; see \cite[Corollary 1.17]{HaaKra}.

We can similarly define $M_d$-weak amenability by requiring that $1$ be in the $\sigma(M_d(G),X_d(G))$-closure of $\{\varphi\in\C[\Gamma]\ :\ \|\varphi\|_{M_d}\leq C\}$. Then, for all $d\geq 1$, $M_d$-weak amenability implies $M_d$-AP. Again, we will see that $M_{d+1}$-weak amenability implies $M_d$-weak amenability, but we do not know if they are different for $d\geq 2$. Every group is $M_1$-weakly amenable.

For every countable group $G$ acting properly on a finite-dimensional CAT(0) cube complex, Guentner and Higson constructed an analytic family of uniformly bounded representations of $G$ satisfying some very specific properties. Using a result of Valette \cite{Val}, their construction allowed them to conclude that every such $G$ is weakly amenable with Cowling--Haagerup constant 1. This extended previous results for groups acting on trees \cite{Pim,PytSzw,Val2} and right-angled Coxeter groups \cite{Jan}.

A key point in Valette's result \cite{Val} is the fact that every coefficient of a uniformly bounded representation of a group $G$ defines an element of $M_2(G)$. Since this is also true for $M_d(G)$, we are able to adapt Valette's proof to this new setting. Combining this with Guentner and Higson's construction, we obtain the following.

\begin{thm}\label{Thm_2}
Let $G$ be a countable group acting properly on a finite-dimensional CAT(0) cube complex. Then $G$ is $M_d$-weakly amenable for all $d\geq 1$.
\end{thm}

In particular, the Dixmier problem has a positive answer for all such groups. We must point out, however, that this fact can also be deduced (after imposing a bound on the orders of finite subgroups) from the Tits alternative for CAT(0) cube complexes \cite{SagWis}, and the fact that groups containing free subgroups are not unitarisable; see e.g. \cite[Chapter 2]{Pis2}.

Finally, we explore some ways of constructing new groups with $M_d$-AP. We show that this property is inherited by subgroups. Moreover, we show that, if $G/\Gamma$ has $M_d$-AP, where $\Gamma\lhd G$ is a normal amenable subgroup, then $G$ has $M_d$-AP. This fact allows us to give examples of non-weakly amenable groups satisfying $M_d$-AP for all $d\geq 1$, including some wreath products and $\operatorname{SL}(2,\Z)\ltimes\Z^2$. See Corollaries \ref{Cor_SL2} and \ref{Cor_wreath}.

We conclude this introduction with a few observations. In \cite{Pis}, Pisier showed that
\begin{align*}
M_{d+1}(\F_\infty)\neq M_d(\F_\infty),\quad\forall d\geq 1,
\end{align*}
where $\F_\infty$ is the free group on countably many generators. This suggests that $M_{d+1}$-AP might be strictly stronger than $M_d$-AP. However, since $\F_\infty$ has a proper action on a tree, Theorem \ref{Thm_2} says that it satisfies $M_d$-AP for all $d\geq 1$. Hence, the following question seems quite natural.

\noindent{\bf Question:} Is $M_d$-AP equivalent to AP for all $d\geq 2$?

A positive answer to this question would solve the Dixmier problem for groups with AP.

\subsection{Organisation of the paper}
In Section \ref{Sec_Gen_mult}, we define $M_d$-weak amenability and $M_d$-AP, and prove Theorem \ref{Thm_1}. In Section \ref{Sec_an_fam}, we focus on holomorphic families of representations and prove Theorem \ref{Thm_2}. Finally, we study some stability properties of $M_d$-AP in Section \ref{Sec_stab}, and construct examples of non-weakly amenable groups satisfying $M_d$-AP.

\section{{\bf Subalgebras of multipliers and approximation properties}}\label{Sec_Gen_mult}

Throughout this paper, $G$ will always denote a discrete group. The group algebra $\C[G]$ is the space of finitely supported functions $g:G\to\C$. For Hilbert spaces $H,K$, we denote by $\mathcal{B}(H,K)$ the space of bounded linear operators from $H$ to $K$. When $H=K$, we will write $\mathcal{B}(H)=\mathcal{B}(H,H)$.

Now we will define the subalgebras of Herz--Schur multipliers introduced by Pisier. We refer the reader to \cite{Pis} for a much more detailed treatment and references.

Let $G$ be a group and let $d\geq 1$ be an integer. We define $M_d(G)$ as the space of all functions $\varphi:G\to\C$ such that there are bounded maps $\xi_i:G\to\mathcal{B}(H_i,H_{i-1})$ ($i=1,\ldots,d$), where $H_i$ is a Hilbert space, $H_0=H_d=\C$, and
\begin{align}\label{phi_xi}
\varphi(t_1\cdots t_d)=\xi_1(t_1)\cdots \xi_d(t_d),\quad\forall t_1,\ldots,t_d\in G.
\end{align}
We define
\begin{align*}
\|\varphi\|_{M_d(G)}=\inf\left\{\sup_{t_1\in G}\|\xi_1(t_1)\|\cdots\sup_{t_d\in G}\|\xi_d(t_d)\|\right\},
\end{align*}
where the infimum runs over all possible decompositions as in \eqref{phi_xi}.

As observed by Pisier, $M_d(G)$ is a dual space. We define $X_d(G)$ as the completion of $\C[G]$ for the norm
\begin{align}\label{norm_X_d}
\|g\|_{X_d(G)}=\sup\left\{\left|\sum_{t\in G}\varphi(t)g(t)\right|\ :\ \varphi\in M_d(G),\ \|\varphi\|_{M_d(G)}\leq 1\right\}.
\end{align}
We have $X_d(G)^*=M_d(G)$ for the duality
\begin{align*}
\langle\varphi,g\rangle=\sum_{t\in G}\varphi(t)g(t),\quad\forall \varphi\in M_d(G),\ \forall g\in\C[G].
\end{align*}

We will say that $G$ is $M_d$\textit{-weakly amenable} if there is $C\geq 1$ such that the constant function $1$ is in the $\sigma(M_d(G),X_d(G))$-closure of the set
\begin{align*}
\{\varphi\in\C[\Gamma]\ :\ \|\varphi\|_{M_d}\leq C\}.
\end{align*}
Observe that $M_2(G)$ is the algebra of Herz--Schur multipliers of $G$. Thus, $M_2$-weak amenability and weak amenability are the same; see e.g. \cite[Theorem 1.12]{HaaKra}.

Let $\Phi:E\to F$ be a bounded linear map between Banach spaces. We denote by $\Phi^*:F^*\to E^*$ its adjoint map. We will make extensive use of the following classical result, whose proof we include for the reader's convenience.

\begin{lem}
Let $E,F$ be Banach spaces and let $\Psi:F^*\to E^*$ be a bounded linear map. Then $\Psi$ is $\sigma(F^*,F)$-$\sigma(E^*,E)$-continuous if and only if there exists a bounded linear map $\Phi:E\to F$ such that $\Psi=\Phi^*$.
\end{lem}
\begin{proof}
Assume first that $\Psi=\Phi^*$. Then the weak*-weak*-continuity follows directly from the identity
\begin{align*}
\langle\Psi f,x\rangle=\langle f,\Phi x\rangle,\quad \forall f\in F^*,\ \forall x\in E.
\end{align*}
Conversely, assume that $\Psi$ is $\sigma(F^*,F)$-$\sigma(E^*,E)$-continuous. Recall that we can view $E$ as a subspace of $E^{**}$. Define $\Phi=\Psi^*|_E:E\to F^{**}$. Thus
\begin{align}\label{iden_Phi}
\langle \Phi x,f\rangle=\langle\Psi f,x\rangle,\quad \forall f\in F^*,\ \forall x\in E,
\end{align}
which shows that $\Phi x$ is $\sigma(F^*,F)$-continuous for all $x\in E$. Therefore $\Phi$ takes values in $F$; see e.g. \cite[\S 3.14]{Rud}. Then the fact that $\Phi^*=\Psi$ also follows from the identity \eqref{iden_Phi}.
\end{proof}

\begin{lem}\label{Lem_M_d+1}
Let $d\geq 1$. If $G$ is $M_{d+1}$-weakly amenable, then it is $M_d$-weakly amenable.
\end{lem}
\begin{proof}
Observe that the inclusion $\iota:M_{d+1}(G)\hookrightarrow M_d(G)$ is contractive. Moreover, by the definition of the norm \eqref{norm_X_d}, the identity map $\C[G]\to\C[G]$ extends to a contractive map $X_d(G)\to X_{d+1}(G)$ whose dual map is $\iota$. This shows that $\iota$ is weak*-weak* continuous. The result follows directly from this.
\end{proof}

We will say that $G$ has the $M_d$-AP if the constant function $1$ is in the $\sigma(M_d(G),X_d(G))$-closure of $\C[G]$. It readily follows that every $M_d$-weakly amenable group has the $M_d$-AP. By the same arguments as before, we obtain the following.

\begin{lem}
Let $d\geq 1$. If $G$ has the $M_{d+1}$-AP, then it has the $M_d$-AP.
\end{lem}

Now we will look at the Fourier--Stieltjes algebra to establish the link between these properties and amenability. The main reference here is \cite{Eym}. See also \cite[Appendix F]{Run}. Let $G$ be a group and let $\pi:G\to\mathcal{B}(H)$ be a unitary representation. We say that $\varphi:G\to\C$ is a \textit{coefficient} of $\pi$ if there exist $\xi,\eta\in H$ such that
\begin{align}\label{coeff_rep}
\varphi(t)=\langle\pi(t)\xi,\eta\rangle,\quad\forall t\in G.
\end{align}
The Fourier--Stieltjes algebra $B(G)$ is the space of all coefficients of unitary representations of $G$. For $\varphi\in B(G)$, we define
\begin{align*}
\|\varphi\|_{B(G)}=\inf\{\|\xi\| \|\eta\|\},
\end{align*}
where the infimum runs over all decompositions as in \eqref{coeff_rep}.

The full C${}^*$-algebra $C^*(G)$ is the completion of $\C[G]$ for the norm
\begin{align*}
\|g\|_{C^*(G)}=\sup\left\{\left\|\sum_{t\in G} g(t)\pi(t)\right\|\ :\ \pi:G\to\mathcal{B}(H)\text{ unitary rep.}\right\}.
\end{align*}
Then $C^*(G)^*=B(G)$ with
\begin{align*}
\langle\varphi,g\rangle=\sum_{t\in G}\varphi(t)g(t),\quad\forall\varphi\in B(G),\ \forall g\in\C[G].
\end{align*}

The Fourier algebra $A(G)$ is the subalgebra of $B(G)$ given by coefficients of the left regular representation. It can be defined as the completion of $\C[G]$ in $B(G)$. Leptin showed \cite{Lep} that $G$ is amenable if and only if $A(G)$ has a bounded approximate identity. We will use a reformulation of this result in terms of the weak* topology $\sigma(B(G),C^*(G))$; see \cite[Theorem 1.4.1]{Run} and \cite[Proposition 1.4.2]{Run}.

\begin{prop}\label{Prop_A_dense}
Let $G$ be a group. Then $G$ is amenable if and only if the constant function $1$ is in the $\sigma(B(G),C^*(G))$-closure of $\{\varphi\in A(G)\ :\ \|\varphi\|_{B(G)}\leq 1\}$ in $B(G)$.
\end{prop}

\begin{rmk}\label{Rmk_C_dense}
Since $A(G)$ is the norm closure of $\C[G]$ in $B(G)$, Proposition \ref{Prop_A_dense} can be restated as follows: $G$ is amenable if and only if the constant function $1$ is in the $\sigma(B(G),C^*(G))$-closure of $\{\varphi\in \C[G]\ :\ \|\varphi\|_{B(G)}\leq 1\}$ in $B(G)$.
\end{rmk}

\begin{cor}\label{Cor_amen_M_d}
Every amenable group is $M_d$-weakly amenable for all $d\geq 1$.
\end{cor}
\begin{proof}
The proof is essentially the same as that of Lemma \ref{Lem_M_d+1}. The inclusion $\iota:B(G)\hookrightarrow M_d(G)$ is contractive because every coefficient of a unitary representation $\varphi=\langle\pi(\,\cdot\,)\xi,\eta\rangle$ satisfies
\begin{align}\label{phi=<pi>}
\varphi(t_1\cdots t_d)=\langle\pi(t_1)\cdots\pi(t_d)\xi,\eta\rangle,\quad\forall t_1,\ldots,t_d\in G.
\end{align}
Moreover, for every $g\in\C[G]$,
\begin{align*}
\|g\|_{X_d(G)} &=\sup\left\{\left|\sum_{t\in G}\varphi(t)g(t)\right|\ :\ \varphi\in M_d(G),\ \|\varphi\|_{M_d(G)}\leq 1\right\}\\
&\geq\sup\left\{\left|\sum_{t\in G}\varphi(t)g(t)\right|\ :\ \varphi\in B(G),\ \|\varphi\|_{B(G)}\leq 1\right\}\\
&=\|g\|_{C^*(G)}.
\end{align*}
Thus the identity map $\C[G]\to\C[G]$ extends to a contractive map $X_d(G)\to C^*(G)$ whose dual map is $\iota$. This shows that $\iota$ is weak*-weak* continuous. Then the result follows from Proposition \ref{Prop_A_dense} and Remark \ref{Rmk_C_dense}.
\end{proof}

Now we are ready to give the proof of Theorem \ref{Thm_1}.

\begin{proof}[Proof of Theorem \ref{Thm_1}]
By Theorem \ref{Thm_Pis}, we have $M_d(G)=B(G)$ for some $d\geq 1$. Observe that this implies that $X_d(G)$ is isomorphic to the full C${}^*$-algebra $C^*(G)$. Since $G$ has the $M_d$-AP, the constant function $1$ is in the $\sigma(B(G),C^*(G))$-closure of $\C[G]$. By Proposition \ref{Prop_A_dense} and Remark \ref{Rmk_C_dense}, $G$ is amenable.
\end{proof}

\section{{\bf Analytic families of representations}}\label{Sec_an_fam}
A representation $\pi:G\to\mathcal{B}(H)$ is said to be \textit{uniformly bounded} if
\begin{align*}
|\pi|=\sup_{t\in G}\|\pi(t)\| <\infty.
\end{align*}
Observe that $\pi$ is unitary if and only if $|\pi|=1$. If $\varphi=\langle\pi(\,\cdot\,)\xi,\eta\rangle$ is a coefficient of a uniformly bounded representation, then $\varphi$ belongs to $M_d(G)$ for all $d\geq 1$ and
\begin{align*}
\|\varphi\|_{M_d(G)}\leq |\pi|^d\|\xi\|\|\eta\|,
\end{align*}
because it can be decomposed as in \eqref{phi=<pi>}.

Let $\Omega$ be an open subset of $\C$ and let $E$ be a Banach space. We say that a function $f:\Omega\to E$ is \textit{holomorphic} if the limit
\begin{align*}
\lim_{z\to z_0}\frac{f(z)-f(z_0)}{z-z_0}
\end{align*}
exists in norm for all $z_0\in\Omega$. See \cite[\S 3.3]{dCaHaa} for different characterisations of Banach space valued holomorphic functions. We will say that a family of representations $\pi_z:G\to\mathcal{B}(H)$ ($z\in\Omega$) on a Hilbert space $H$ is holomorphic if, for every $t\in G$, the function
\begin{align*}
z\in\Omega\longmapsto \pi_z(t)\in\mathcal{B}(H)
\end{align*}
is holomorphic.

\begin{lem}\label{Lem_holom}
Let $\pi_z:G\to\mathcal{B}(H)$ ($z\in\Omega$) be a holomorphic family of uniformly bounded representations. Assume moreover that $z\mapsto |\pi_z|$ is bounded on compact subsets of $\Omega$. Then, for all $\xi,\eta\in H$ and $d\geq 1$, the function
\begin{align*}
z\in\Omega\longmapsto \langle\pi_z(\,\cdot\,)\xi,\eta\rangle \in M_d(G)
\end{align*}
is holomorphic.
\end{lem}
\begin{proof}
Let $\psi_z(t)=\langle\pi_z(t)\xi,\eta\rangle$. By the discussion above, $\psi_z$ is an element of $M_d(G)$ with
\begin{align*}
\|\psi_z\|_{M_d(G)}\leq |\pi_z|^d\|\xi\|\|\eta\|,
\end{align*}
for all $z\in\Omega$. Moreover, for all $t\in G$, the function $z\mapsto\psi_z(t)$ is holomorphic in the usual sense; see \cite[\S 3.3]{dCaHaa}. Now observe that
\begin{align*}
\psi_z(t)=\langle\psi_z,\delta_t\rangle_{M_d(G),X_d(G)},\quad\forall t\in G.
\end{align*}
Therefore, by linearity, for every $g\in\C[G]$, the function
\begin{align*}
z\in\Omega\longmapsto\langle\psi_z,g\rangle_{M_d(G),X_d(G)}
\end{align*}
is holomorphic. Now let $g\in X_d(G)$ and let $(g_n)$ be a sequence in $\C[G]$ such that $\|g-g_n\|_{X_d(G)}\to 0$. Then
\begin{align*}
|\langle\psi_z,g\rangle-\langle\psi_z,g_n\rangle| &\leq \|\psi_z\|_{M_d(G)}\|g-g_n\|_{X_d(G)}\\
&\leq |\pi_z|^d\|\xi\|\|\eta\|\|g-g_n\|_{X_d(G)}.
\end{align*}
Since $|\pi_z|$ is bounded on compact subsets of $\Omega$, the convergence is uniform on such sets, and therefore the limit is analytic. We conclude that
\begin{align*}
z\in\Omega\longmapsto\langle\psi_z,g\rangle_{M_d(G),X_d(G)}
\end{align*}
is holomorphic for all $g\in X_d(G)$, which means that $z\mapsto\psi_z$ is holomorphic; see \cite[\S 3.3]{dCaHaa}.
\end{proof}

Valette \cite{Val} devised a way of proving weak amenability by using holomorphic families of uniformly bounded representations. The following result is a simple adaptation of \cite{Val} to our new setting. We will denote by $\D$ the open unit disk of $\C$.

\begin{prop}\label{Prop_hol_fam}
Let $G$ be a countable discrete group with identity element $e$. Assume there is a proper function $\ell:G\to\N$ such that $\ell(e)=0$, and a holomorphic family $\{\pi_z\}_{z\in\D}$ of uniformly bounded representations on a Hilbert space $H$ such that $\pi_r$ is unitary for $0<r<1$, and $z\mapsto |\pi_z|$ is bounded on compact subsets of $\D$. If there is $\xi\in H$ such that
\begin{align*}
z^{\ell(t)}=\langle\pi_z(t)\xi,\xi\rangle,\quad\forall z\in\D,\ \forall t\in G,
\end{align*}
then $G$ is $M_d$-weakly amenable for every $d\geq 1$.
\end{prop}
\begin{proof}
Let us fix $d\geq 1$ and define
\begin{align*}
\psi_z(t)=z^{\ell(t)},\quad\forall z\in\D,\ \forall t\in G.
\end{align*}
By Lemma \ref{Lem_holom}, $z\mapsto\psi_z$ is a holomorphic map from $\D$ to $M_d(G)$. Let $S^1$ denote the unit circle, and let $F_N:S^1\to\R$ be the F\'ejer kernel:
\begin{align*}
F_N(z)=\sum_{|n|\leq N}\left(1-\frac{|n|}{N+1}\right)z^n,\quad\forall N\in\N,\ \forall z\in S^1.
\end{align*}
For $0<r<1$ and $N\in\N$, we define
\begin{align*}
\Phi_{N,r}=\frac{1}{2\pi}\int_0^{2\pi}F_N\left(e^{\ii\theta}\right)\psi_{re^{\ii\theta}}\, d\theta.
\end{align*}
Then $\Phi_{N,r}\in\C[G]$,
\begin{align*}
\lim_{N\to\infty}\|\Phi_{N,r}\|_{M_d(G)}=1,\quad\forall r\in(0,1),
\end{align*}
and
\begin{align*}
\lim_{r\to 1}\lim_{N\to\infty}\Phi_{N,r}(t)=1,\quad\forall t\in G;
\end{align*}
see \cite{Val} for details. This shows that, for every $\varepsilon>0$, we can find a sequence $(\varphi_n)$ in $\C[G]$, with $\|\varphi_n\|_{M_d(G)}\leq 1+\varepsilon$, converging pointwise to $1$. Since $\C[G]$ is dense in $X_d(G)$, pointwise convergence plus uniform boundedness yield weak* convergence. Therefore the constant function $1$ is in the $\sigma(M_d(G),X_d(G))$-closure of the set
\begin{align*}
\{\varphi\in\C[G]\ :\ \|\varphi\|_{M_d(G)}\leq 1+\varepsilon\},
\end{align*}
for every $\varepsilon>0$. We conclude that $G$ is $M_d$-weakly amenable.
\end{proof}

\begin{proof}[Proof of Theorem \ref{Thm_2}]
Since $G$ acts properly on a finite-dimensional CAT(0) cube complex, the construction in \cite{GueHig} shows that $G$ satisfies the hypotheses of Proposition \ref{Prop_hol_fam}. The uniform bound on compact subsets is given by \cite[Proposition 5.3]{GueHig}. We conclude that $G$ is $M_d$-weakly amenable for every $d\geq 1$.
\end{proof}

\section{{\bf Stability properties}}\label{Sec_stab}
In this section we study how the $M_d$-AP behaves with respect to certain group operations.

\begin{lem}
Let $d\geq 1$. Let $G$ be a group and $\Gamma\subset G$ a subgroup. If $G$ has the $M_d$-AP, then so does $\Gamma$.
\end{lem}
\begin{proof}
By the definition of the norm of $M_d$, the restriction map
\begin{align*}
\varphi\in M_d(G)\mapsto \varphi|_\Gamma\in M_d(\Gamma)
\end{align*}
is contractive. Moreover, a simple calculation shows that this is the dual map of
\begin{align*}
g\in X_d(\Gamma)\mapsto\tilde{g}\in X_d(G),
\end{align*}
defined by
\begin{align}\label{def_gtild}
\tilde{g}(t)=\begin{cases}
g(t), & t\in\Gamma\\ 0 & t\in G\setminus\Gamma,
\end{cases}
\end{align}
which is also contractive. Therefore the restriction map $\varphi\mapsto \varphi|_\Gamma$ is weak*-weak* continuous and the result follows.
\end{proof}

\begin{rmk}
We do not know if $M_d$-AP is stable under direct products.
\end{rmk}

Now we turn to extensions of groups. The proof of the following lemma is inspired by \cite[Theorem 1.15]{HaaKra}.

\begin{lem}\label{Lem_M_d_ext}
Let $G$ be a group and $\Gamma\lhd G$ a normal amenable subgroup. If $G/\Gamma$ has the $M_d$-AP, then so does $G$.
\end{lem}
\begin{proof}
For each $f\in\C[G]$, let us define a map $\Phi_f:C^*(G)\to C^*(G)$ by
\begin{align*}
\Phi_f(g)=f\ast g,\quad\forall g\in\C[G].
\end{align*}
Observe that this is simply the product in the C${}^*$-algebra $C^*(G)$, so $\Phi_f$ is well defined and $\|\Phi_f\|=\|f\|_{C^*(G)}$. Now let $P:C^*(G)\to C^*(\Gamma)$ be the extension of the restriction map: $g\in\C[G]\mapsto g|_\Gamma\in\C[\Gamma]$. Then $P$ is a contraction; see e.g. \cite[Proposition 3.5]{Pis3}. Defining $\Psi_f=P\circ\Phi_f$, we get a bounded map from $C^*(G)$ to $C^*(\Gamma)$ such that
\begin{align*}
\Psi_f(g)=\left.\left(f\ast g\right)\right|_\Gamma,\quad\forall g\in\C[G].
\end{align*}
Then the adjoint map $\Psi_f^*:B(\Gamma)\to B(G)$ is weak*-weak* continuous. A simple calculation shows that
\begin{align*}
\Psi_f^*(\varphi)=\check{f}\ast\tilde{\varphi},\quad\forall\varphi\in B(\Gamma),
\end{align*}
where $\check{f}(t)=f(t^{-1})$, and $\tilde{\varphi}\in B(G)$ is defined as in \eqref{def_gtild}. Since $\Gamma$ is amenable, by Proposition \ref{Prop_A_dense}, there is a net $(\varphi_i)$ in $\C[\Gamma]$ converging to $1$ in $\sigma(B(\Gamma),C^*(\Gamma))$. Thus, $\Psi_f^*(\varphi_i)$ converges to $\check{f}\ast\mathds{1}_\Gamma$ in $\sigma(B(G),C^*(G))$, where $\mathds{1}_\Gamma$ is the indicator function of $\Gamma$ in $G$. As we saw in the proof of Corollary \ref{Cor_amen_M_d}, the inclusion $B(G)\hookrightarrow M_d(G)$ is weak*-weak* continuous, hence
\begin{align*}
\check{f}\ast\tilde{\varphi}_i\to\check{f}\ast\mathds{1}_\Gamma\quad\text{in } \sigma(M_d(G),X_d(G)).
\end{align*}
Therefore
\begin{align}\label{incl_w*cl}
\left\{ f\ast\mathds{1}_\Gamma\ :\ f\in\C[G]\right\} \subset \overline{\C[G]}^{\sigma(M_d(G),X_d(G))}.
\end{align}
Thus, we only need to show that the constant function $1$ is in the $\sigma(M_d(G),X_d(G))$-closure of the left hand side of \eqref{incl_w*cl}. Let $q:G\to G/\Gamma$ be the quotient map, and let us define $T:\C[G]\to\C[G/\Gamma]$ by
\begin{align}\label{def_T}
T(f)(q(t))=f\ast\mathds{1}_\Gamma(t)=\sum_{s\in\Gamma}f(ts),\quad\forall f\in\C[G],\ \forall t\in G.
\end{align}
Observe that this is well defined since $f\ast\mathds{1}_\Gamma$ is constant on each coset. Let $\Theta:M_d(G/\Gamma)\to M_d(G)$ be the contraction given by $\Theta(\psi)=\psi\circ q$ (as observed by Pisier, this is actually an isometry, which can be shown by choosing a lifting $G/\Gamma\to G$). Now observe that, for all $f\in\C[G]$,
\begin{align*}
\langle\Theta(\psi),f\rangle &= \sum_{t\in G} \psi(q(t))f(t)\\
&= \sum_{x\in G/\Gamma}\psi(x)\sum_{s\in\Gamma}f(\sigma(x)s)\\
&= \langle\psi,T(f)\rangle,
\end{align*}
where $\sigma:G/\Gamma\to G$ is any lifting. This shows that $T$ extends to a bounded map $T:X_d(G)\to X_d(G/\Gamma)$ and that $T^*=\Theta$. Now we will use the fact that $G/\Gamma$ has the $M_d$-AP. Let $(\psi_i)$ be a net in $\C[G/\Gamma]$ converging to $1$ in $\sigma(M_d(G/\Gamma),X_d(G/\Gamma))$. Again, taking a lifting $G/\Gamma\to G$, one can see that the map \eqref{def_T} is onto. So, for each $i$, there exists $f_i\in\C[G]$ such that $T(f_i)=\psi_i$. In other words,
\begin{align*}
\Theta(\psi_i)=\psi_i\circ q=f_i\ast\mathds{1}_\Gamma.
\end{align*}
Since $\Theta$ is weak*-weak* continuous, we know that $\Theta(\psi_i)$ converges to $1$ in $\sigma(M_d(G),X_d(G))$. Hence, by \eqref{incl_w*cl}, the constant function $1$ is in the $\sigma(M_d(G),X_d(G))$-closure of $\C[G]$. We conclude that $G$ has the $M_d$-AP.
\end{proof}

This lemma allows us to give examples of non-weakly amenable groups satisfying $M_d$-AP for all $d\geq 1$. Recall that, if a group $\Gamma$ acts by automorphisms on another group $G$, we can define the semidirect product $\Gamma\ltimes G$, which coincides with $\Gamma\times G$ as a set, but the group operation is twisted by the action $\Gamma\curvearrowright G$. We get an exact sequence:
\begin{align*}
1\to G\to \Gamma\ltimes G\to\Gamma\to 1.
\end{align*}

\begin{cor}\label{Cor_SL2}
The group $\operatorname{SL}(2,\Z)\ltimes\Z^2$ has $M_d$-AP for all $d\geq 1$.
\end{cor}
\begin{proof}
Let $G=\operatorname{SL}(2,\Z)\ltimes\Z^2$. We have an exact sequence:
\begin{align*}
1\to\Z^2\to G\to\operatorname{SL}(2,\Z)\to 1.
\end{align*}
So $\Z^2$ is a normal amenable subgroup of $G$. On the other hand, $\operatorname{SL}(2,\Z)$ has a proper isometric action on a tree (which is a $1$-dimensional CAT(0) cube complex); see e.g. \cite[\S 4.2]{Ser}. By Theorem \ref{Thm_2}, it has $M_d$-AP for all $d\geq 1$. Therefore, Lemma \ref{Lem_M_d_ext} tells us that $G$ has $M_d$-AP for all $d\geq 1$.
\end{proof}

As shown by Haagerup \cite{Haa}, $\operatorname{SL}(2,\Z)\ltimes\Z^2$ is not weakly amenable. See also \cite{Oza} for a different proof of this fact. Thus, having $M_d$-AP for all $d\geq 1$ does not imply weak amenability. As a matter of fact, the results in \cite{Oza} allow us to give many more examples in this direction. Let $\Gamma$ and $G$ be two groups, and let $\bigoplus_G \Gamma$ be the direct sum of copies of $\Gamma$ indexed by $G$. Then $G$ has an action on $\bigoplus_G \Gamma$ by shifts. The wreath product $\Gamma\wr G$ is defined as the semidirect product
\begin{align*}
G\ltimes\bigoplus_G \Gamma.
\end{align*}

\begin{cor}\label{Cor_wreath}
Let $\Gamma$ be an amenable group and let $G$ be a group with $M_d$-AP. Then the wreath product $\Gamma\wr G$ has $M_d$-AP.
\end{cor}
\begin{proof}
Again, we have an exact sequence:
\begin{align*}
1\to\bigoplus_G\Gamma\to \Gamma\wr G\to G\to 1.
\end{align*}
Since $\bigoplus_G\Gamma$ is a direct sum of amenable groups, it is also amenable. On the other hand, $G$ has $M_d$-AP. Therefore, by Lemma \ref{Lem_M_d_ext}, $\Gamma\wr G$ has $M_d$-AP.
\end{proof}

Ozawa showed that, if $\Gamma$ is not trivial and $G$ is not amenable, then $\Gamma\wr G$ is not weakly amenable; see \cite[Corollary 4]{Oza}. On the other hand, by Theorem \ref{Thm_2}, we know that there are many examples of non-amenable groups satisfying $M_d$-AP for all $d\geq 1$. Thus, Corollary \ref{Cor_wreath} gives us many examples of non-weakly amenable groups satisfying $M_d$-AP for all $d\geq 1$.

\bibliographystyle{plain} 

\bibliography{Bibliography}

\end{document}